\documentclass{amsart}
\usepackage{amssymb,latexsym,amsmath,amscd,graphicx,graphics,epic,eepic,bm,color,array}
\usepackage{enumerate}
\usepackage[all,knot,poly]{xy}

\newcommand{\zed}{\mathbb{Z}}

\newcommand{\C}{\mathbb{C}}

\newcommand{\ve}{\varepsilon}

\theoremstyle{plain}
\newtheorem{theorem}{Theorem}[section]
\newtheorem{lemma}[theorem]{Lemma}
\newtheorem{proposition}[theorem]{Proposition}
\newtheorem{corollary}[theorem]{Corollary}
\newtheorem{conjecture}[theorem]{Conjecture}
\newtheorem{question}[theorem]{Question}

\theoremstyle{definition}
\newtheorem{definition}[theorem]{Definition}

\newtheorem{acknowledgments}{Acknowledgments\ignorespaces}

\theoremstyle{remark}
\newtheorem{remark}[theorem]{Remark}

\numberwithin{equation}{section}

\begin{document}

\title{Colored Morton-Franks-Williams inequalities}

\author{Hao Wu}

\address{Department of Mathematics, The George Washington University, Monroe Hall, Room 240, 2115 G Street, NW, Washington DC 20052}

\email{haowu@gwu.edu}



\begin{abstract}
We generalize the Morton-Franks-Williams inequality \cite{FW,Mo} to the colored $\mathfrak{sl}(N)$ link homology defined in \cite{Wu-color}, which gives infinitely many new bounds for the braid index and the self linking number. A key ingredient of our proof is a composition product for the general MOY graph polynomial, which generalizes that of Wagner \cite{Wagner-composition}.
\end{abstract}

\maketitle

\section{Introduction}\label{sec-intro}

\subsection{The Morton-Franks-Williams inequality}\label{subsec-original-Morton-Franks-Williams} 

The HOMFLY-PT polynomial \cite{HOMFLY,PT} is an invariant for oriented links in $S^3$ in the form of a two variable polynomial $\mathsf{P}$. In this paper, we use the following normalization of the HOMFLY-PT polynomial:
\[
\begin{cases}    
x\mathsf{P}(\setlength{\unitlength}{.5pt}
\begin{picture}(65,20)(50,0)
\put(100,-20){\vector(-1,1){40}}

\put(60,-20){\line(1,1){15}}

\put(85,5){\vector(1,1){15}}

\end{picture})-x^{-1}\mathsf{P}(\setlength{\unitlength}{.5pt}
\begin{picture}(65,20)(-110,0)
\put(-100,-20){\vector(1,1){40}}

\put(-60,-20){\line(-1,1){15}}

\put(-85,5){\vector(-1,1){15}}

\end{picture})= y\mathsf{P}(\setlength{\unitlength}{.5pt}
\begin{picture}(65,20)(50,0)
\put(100,-20){\vector(0,1){40}}

\put(60,-20){\vector(0,1){40}}

\end{picture}), &\\
   & \\
\mathsf{P}(\text{unknot})=\frac{x-x^{-1}}{y}. &
\end{cases}
\]
Morton \cite{Mo} and independently, Franks, Williams \cite{FW} established the Morton-Franks-Williams inequality, which states that, for a closed braid $B$ with writhe $w$ and $b$ strands, 
\begin{equation}\label{ineq-MFW}
w-b \leq \min\deg_x \mathsf{P}(B) \leq \max\deg_x \mathsf{P}(B) \leq w+b.
\end{equation}
If we consider the $\mathfrak{sl}(N)$ HOMFLY-PT polynomial 
\[
\mathsf{P}_N=\mathsf{P}|_{x=q^N,~y=q-q^{-1}},
\]
then the Morton-Franks-Williams inequality \eqref{ineq-MFW} can be expressed as
\begin{equation}\label{ineq-MWF-sl-N}
w-b \leq \lim_{N\rightarrow \infty}\frac{\min\deg_q \mathsf{P}_N(B)}{N-1} \leq \lim_{N\rightarrow \infty}\frac{\max\deg_q \mathsf{P}_N(B)}{N-1} \leq w+b.
\end{equation}

Khovanov and Rozansky \cite{KR1} categorified the $\mathfrak{sl}(N)$ HOMFLY-PT polynomial. That is, they constructed an invariant $\zed^{\oplus2}$-graded homology $H_N^{i,j}$ for oriented links such that, for any oriented link $L$, the graded Euler characteristic of this homology is 
\[
\sum_{i,j}(-1)^i q^j \dim H_N^{i,j}(L) = \mathsf{P}_N(L).
\]
We call $i$ the homological grading of $H_N$ and $j$ the quantum grading of $H_N$. Dunfield, Gukov, Rasmussen \cite{Dunfield-Gukov-Rasmussen} and independently, myself \cite{Wu5} refined \eqref{ineq-MWF-sl-N} to
\begin{equation}\label{ineq-MWF-sl-N-KR}
w-b \leq \liminf_{N\rightarrow \infty}\frac{\min\deg_q H_N(B)}{N-1} \leq \limsup_{N\rightarrow \infty}\frac{\max\deg_q H_N(B)}{N-1} \leq w+b,
\end{equation}
where $\min\deg_q H_N(B)$ and $\max\deg_q H_N(B)$ are the minimal and maximal non-vanishing quantum degrees of $H_N(B)$.

Recall that $w-b$ is the self linking number of the braid $B$. So the above inequalities provide upper bounds for the maximal self linking number of a given link. It is easy to see that these inequalities also provide lower bounds for the braid index of a given link. For more detailed reviews about these and related inequalities, please see \cite{Fer,Ng-bennequin}.

\subsection{Colored Morton-Franks-Williams inequalities}\label{subsec-colored-Morton-Franks-Williams} 

In \cite{Wu-color}, I generalized Khovanov and Rozansky's construction to categorify the Reshetikhin-Turaev $\mathfrak{sl}(N)$ polynomial for links colored by wedge powers of the defining representation of $\mathfrak{sl}(N;\C)$. That is, I constructed an invariant $\zed^{\oplus2}$-graded homology $H_N^{i,j}$ for oriented links colored by wedge powers of the defining representation of $\mathfrak{sl}(N;\C)$ such that, for any such colored oriented link $L$, the graded Euler characteristic of this homology is 
\[
\sum_{i,j}(-1)^i q^j \dim H_N^{i,j}(L) = \mathsf{P}_N(L),
\]
where $\mathsf{P}_N(L)$ is the Reshetikhin-Turaev $\mathfrak{sl}(N)$ polynomial of $L$. Again, we call $i$ the homological grading of $H_N$ and $j$ the quantum grading of $H_N$. Moreover, for simplicity, instead of saying something is colored by the $m$-fold wedge power of the defining representation of $\mathfrak{sl}(N;\C)$, we simply say that it is colored by $m$. If all components of a link are colored by $1$, then we say that the link is uncolored. The Reshetikhin-Turaev $\mathfrak{sl}(N)$ polynomial of an uncolored link is the $\mathfrak{sl}(N)$ HOMFLY-PT polynomial of this link. The colored $\mathfrak{sl}(N)$ link homology of an uncolored link is isomorphic to its Khovanov-Rozansky $\mathfrak{sl}(N)$ homology.

The goal of the present paper is to show that the Morton-Franks-Williams inequality generalizes to the colored $\mathfrak{sl}(N)$ link homology. We do so by establishing the following technical result.

\begin{proposition}\label{Prop-colored-MFW-sl-N}
Let $B$ be a closed braid with $b$ strands, $l_+$ positive crossings and $l_-$ negative crossings. Denote by $l=l_++l_-$ the total number of crossings in $B$. Recall that the writhe of $B$ is $w=l_+ - l_-$. For a positive integer $m$, denote by $B^{(m)}$ the colored link obtained by coloring $B$ entirely by $m$. Then, for any $N>m$, 
\begin{equation}\label{eq-colored-MFW-sl-N}
w -b +\frac{w -ml}{N-m} \leq \frac{\min\deg_q H_N(B^{(m)})}{m(N-m)} \leq \frac{\max\deg_q H_N(B^{(m)})}{m(N-m)} \leq w +b +\frac{w +ml}{N-m},
\end{equation}
where $\min\deg_q H_N(B^{(m)})$ and $\max\deg_q H_N(B^{(m)})$ are the minimal and maximal non-vanishing quantum degrees of $H_N(B^{(m)})$.
\end{proposition}

Letting $N\rightarrow \infty$ in \eqref{eq-colored-MFW-sl-N}, we easily get the following colored homological Morton-Franks-Williams inequalities.

\begin{theorem}\label{THM-colored-MFW-sl-N-limit}
Let $B$ be a closed braid with writhe $w$ and $b$ strands. Then
\begin{equation}\label{eq-colored-MFW-sl-N-limit-special}
w -b \leq \liminf_{N\rightarrow \infty} \frac{\min\deg_q H_N(B^{(m)})}{m(N-m)} \leq  \limsup_{N\rightarrow \infty}\frac{\max\deg_q H_N(B^{(m)})}{m(N-m)} \leq w +b.
\end{equation}
More generally, for any two sequences $\{m_k\}$ and $\{N_k\}$ of positive integers satisfying $\lim_{k\rightarrow \infty} \frac{1}{N_k} = \lim_{k\rightarrow \infty} \frac{m_k}{N_k} =0$, 
\begin{equation}\label{eq-colored-MFW-sl-N-limit-general}
w-b ~\leq~ \liminf_{k\rightarrow+\infty} \frac{\min\deg_q H_{N_k}(B^{(m_k)})}{m_k(N_k-m_k)} ~\leq~ \limsup_{k\rightarrow+\infty} \frac{\min\deg_q H_{N_k}(B^{(m_k)})}{m_k(N_k-m_k)} ~\leq~ w+b.
\end{equation}
\end{theorem}

Since the colored $\mathfrak{sl}(N)$ homology categorifies the corresponding colored Reshetikhin-Turaev $\mathfrak{sl}(N)$ polynomial, Theorem \ref{THM-colored-MFW-sl-N-limit} implies the following colored polynomial Morton-Franks-Williams inequalities.

\begin{corollary}\label{Coro-colored-MFW-sl-N-limit-poly}
Let $B$ be a closed braid with writhe $w$ and $b$ strands. Then
\begin{equation}\label{eq-colored-MFW-sl-N-limit-special-poly}
w -b \leq \liminf_{N\rightarrow \infty} \frac{\min\deg_q \mathsf{P}_N(B^{(m)})}{m(N-m)} \leq  \limsup_{N\rightarrow \infty}\frac{\max\deg_q \mathsf{P}_N(B^{(m)})}{m(N-m)} \leq w +b.
\end{equation}
More generally, for any two sequences $\{m_k\}$ and $\{N_k\}$ of positive integers satisfying $\lim_{k\rightarrow \infty} \frac{1}{N_k} = \lim_{k\rightarrow \infty} \frac{m_k}{N_k} =0$, 
\begin{equation}\label{eq-colored-MFW-sl-N-limit-general-poly}
w-b ~\leq~ \liminf_{k\rightarrow+\infty} \frac{\min\deg_q \mathsf{P}_{N_k}(B^{(m_k)})}{m_k(N_k-m_k)} ~\leq~ \limsup_{k\rightarrow+\infty} \frac{\min\deg_q \mathsf{P}_{N_k}(B^{(m_k)})}{m_k(N_k-m_k)} ~\leq~ w+b.
\end{equation}
\end{corollary}

Clearly, \eqref{eq-colored-MFW-sl-N-limit-special} and \eqref{eq-colored-MFW-sl-N-limit-special-poly} specialize to \eqref{ineq-MWF-sl-N-KR} and \eqref{ineq-MWF-sl-N} when $m=1$. Moreover, Theorem \ref{THM-colored-MFW-sl-N-limit} and Corollary \ref{Coro-colored-MFW-sl-N-limit-poly} give infinitely many new upper bounds for the self linking number and lower bounds for the braid index.

\subsection{A composition product for the MOY graph polynomial} 

To prove the colored Morton-Franks-Williams inequalities, we only need to prove Proposition \ref{Prop-colored-MFW-sl-N}. We do so by generalizing the proof of the Morton-Franks-Williams inequality by Jaeger \cite{Jaeger-composition}. A key ingredient of our proof is a composition product for the MOY graph polynomial, which generalizes that of Wagner \cite{Wagner-composition}. Before stating our composition product, we briefly recall some basic facts about the MOY graph polynomial \cite{MOY}. (A more detailed review will be given in Section \ref{sec-MOY-graph-homology}.)

\begin{figure}[ht]
\[
\xymatrix{
\setlength{\unitlength}{1pt}
\begin{picture}(50,50)(-145,-5)

\put(-120,0){\vector(0,1){20}}

\put(-128,7){\small{$e$}}

\put(-115,7){\tiny{$m+n$}}

\put(-120,20){\vector(1,1){20}}

\put(-108,25){\tiny{$n$}}

\put(-120,20){\vector(-1,1){20}}

\put(-135,25){\tiny{$m$}}

\put(-133,35){\small{$e_1$}}

\put(-115,35){\small{$e_2$}}

\end{picture} && \setlength{\unitlength}{1pt}
\begin{picture}(50,50)(-145,-45)

\put(-120,-20){\vector(0,1){20}}

\put(-128,-7){\small{$e$}}

\put(-115,-7){\tiny{$m+n$}}

\put(-100,-40){\vector(-1,1){20}}

\put(-108,-25){\tiny{$n$}}

\put(-140,-40){\vector(1,1){20}}

\put(-135,-25){\tiny{$m$}}

\put(-133,-38){\small{$e_1$}}

\put(-115,-38){\small{$e_2$}}

\end{picture}
}
\]
\caption{}\label{fig-MOY-vertex}
\end{figure}
 
\begin{definition}\label{def-MOY}
An MOY coloring of an oriented trivalent graph is a function from the set of edges of this graph to the set of non-negative integers such that every vertex of the colored graph is of one of the two types in Figure \ref{fig-MOY-vertex}. 

An MOY graph is an oriented trivalent graph embedded in the plane equipped with an MOY coloring. 

For an MOY graph $\Gamma$, denote by $E(\Gamma)$ the set of all edges of $\Gamma$ and by $V(\Gamma)$ the set of all vertices of $\Gamma$.
\end{definition}

For every positive integer $N$ and every MOY graph $\Gamma$, Murakami, Ohtsuki and Yamada \cite{MOY} defined a single variable polynomial $\left\langle \Gamma \right\rangle_N$, which we call the $\mathfrak{sl}(N)$ MOY graph polynomial. To be consistent with the definition of the colored $\mathfrak{sl}(N)$ link homology in \cite{Wu-color}, we use a slightly different normalization in the present paper. (See Section \ref{sec-MOY-graph-homology} for more details.)

\begin{figure}[ht]
$
\xymatrix{
\setlength{\unitlength}{1pt}
\begin{picture}(100,40)(-50,0)

\put(5,0){\vector(0,1){20}}

\put(-5,0){\vector(0,1){20}}

\put(-25,0){\vector(0,1){20}}

\put(25,0){\vector(0,1){20}}

\put(-5,20){\vector(-1,1){20}}

\put(-25,20){\vector(-1,1){20}}

\put(5,20){\vector(1,1){20}}

\put(25,20){\vector(1,1){20}}

\put(-22,20){{$\dots$}}

\put(10,20){{$\dots$}}

\end{picture} & \text{or} & \setlength{\unitlength}{1pt}
\begin{picture}(100,40)(-50,0)

\put(5,20){\vector(0,1){20}}

\put(-5,20){\vector(0,1){20}}

\put(25,20){\vector(0,1){20}}

\put(-25,20){\vector(0,1){20}}

\put(-25,0){\vector(1,1){20}}

\put(-45,0){\vector(1,1){20}}

\put(25,0){\vector(-1,1){20}}

\put(45,0){\vector(-1,1){20}}

\put(-22,20){$\dots$}

\put(10,20){$\dots$}

\end{picture}
} 
$
\caption{}\label{tri-vertex-split} 

\end{figure}

\begin{definition}\label{def-MOY-rot}
Let $\Gamma$ be an MOY graph. For each edge of $\Gamma$ with color $m$, change it into $m$ parallel edges. Also replace each vertex of $\Gamma$ by one of the two shapes in Figure \ref{tri-vertex-split}. This changes $\Gamma$ into a collection $\mathcal{C}$ of oriented circles embedded in the plane. For each circle $C \in \mathcal{C}$, denote by $\mathrm{rot}(C)$ its usual rotation number, that is
\[
\mathrm{rot}(C) = 
\begin{cases}
1 & \text{if } C \text{ is counterclockwise,} \\
-1 & \text{if } C \text{ is clockwise.}
\end{cases}
\]
Then define
\begin{equation}\label{eq-def-rot-MOY}
\mathrm{rot}(\Gamma) = \sum_{C \in \mathcal{C}} \mathrm{rot}(C).
\end{equation}
\end{definition}

\begin{definition}\label{def-MOY-label}
Let $\Gamma$ be an MOY graph. Denote by $\mathsf{c}$ its color function. That is, for every edge $e$ of $\Gamma$, the color of $e$ is $\mathsf{c}(e)$. A labeling $\mathsf{f}$ of $\Gamma$ is an MOY coloring of the underlying oriented trivalent graph of $\Gamma$ such that $\mathsf{f}(e)\leq \mathsf{c}(e)$ for every edge $e$ of $\Gamma$.

Denote by $\mathcal{L}(\Gamma)$ the set of all labellings of $\Gamma$. For every $\mathsf{f} \in \mathcal{L}(\Gamma)$, denote by $\Gamma_{\mathsf{f}}$ the MOY graph obtained by re-coloring the underlying oriented trivalent graph of $\Gamma$ using $\mathsf{f}$.

For every $\mathsf{f} \in \mathcal{L}(\Gamma)$, define a function $\bar{\mathsf{f}}$ on $E(\Gamma)$ by $\bar{\mathsf{f}}(e)= \mathsf{c}(e)- \mathsf{f}(e)$ for every edge $e$ of $\Gamma$. It is easy to see that $\bar{\mathsf{f}}\in \mathcal{L}(\Gamma)$.

Let $v$ be a vertex of $\Gamma$ of either type in Figure \ref{fig-MOY-vertex}. (Note that, in either case, $e_1$ is to the left of $e_2$ when one looks in the direction of $e$.) For every $\mathsf{f} \in \mathcal{L}(\Gamma)$, define
\[
[v|\Gamma|\mathsf{f}] = \frac{1}{2} (\mathsf{f}(e_1)\bar{\mathsf{f}}(e_2) - \bar{\mathsf{f}}(e_1)\mathsf{f}(e_2)).
\]
\end{definition}

\begin{theorem}\label{THM-composition-product}
Let $\Gamma$ be an MOY graph. For positive integers $M$, $N$ and $\mathsf{f} \in \mathcal{L}(\Gamma)$, define
\[
\sigma_{M,N}(\Gamma,\mathsf{f}) = M \cdot \mathrm{rot}(\Gamma_{\bar{\mathsf{f}}}) - N \cdot \mathrm{rot}(\Gamma_{\mathsf{f}}) +  \sum_{v\in V(\Gamma)} [v|\Gamma|\mathsf{f}].
\] 
Then
\begin{equation}\label{eq-composition-product}
\left\langle \Gamma \right\rangle_{M+N} = \sum_{\mathsf{f} \in \mathcal{L}(\Gamma)} q^{\sigma_{M,N}(\Gamma,\mathsf{f})} \cdot \left\langle \Gamma_{\mathsf{f}} \right\rangle_M \cdot \left\langle \Gamma_{\bar{\mathsf{f}}} \right\rangle_N.
\end{equation}
\end{theorem}

It is straightforward to check that, if $\Gamma$ is a $1,2$-colored MOY graph, then \eqref{eq-composition-product} specializes to Wagner's composition product \cite[Lemma 1.1]{Wagner-composition}, which implies Jaeger's composition product for the HOMFLY-PT polynomial \cite[Proposition 1]{Jaeger-composition}.

\subsection{Open problems and remarks}

Jaeger's composition product \cite{Jaeger-composition} was generalized by Turaev \cite{Turaev-composition} to a comultiplication in the HOMFLY-PT skein module of a thickened surface. Przytycki \cite{Przytycki-composition} further proved that this leads to a Hopf algebra structure on this module.

\begin{question}
Is it possible to interpret the composition product \eqref{eq-composition-product} in a framework similar to that given by Turaev and Przytycki?
\end{question}

In \cite{Wu-color}, I constructed a $\zed$-graded $\mathfrak{sl}(N)$ graph homology $H_N$ whose graded dimension is the $\mathfrak{sl}(N)$ MOY graph polynomial. (See \cite[Theorem 14.7]{Wu-color}.) So Theorem \ref{THM-composition-product} implies the following corollary, which generalizes \cite[Theorem 1.2]{Wagner-composition}.

\begin{corollary}\label{cor-composition-product-homology}
Let $\Gamma$ be an MOY graph. Then, for positive integers $M$ and $N$,
\[
H_{M+N}(\Gamma) \cong \bigoplus_{\mathsf{f} \in \mathcal{L}(\Gamma)} H_M(\Gamma_{\mathsf{f}}) \otimes_\C H_N(\Gamma_{\bar{\mathsf{f}}}) \{q^{\sigma_{M,N}(\Gamma,\mathsf{f})}\},
\]
where the isomorphism preserves the $\zed$-grading.
\end{corollary}

\begin{question}
The isomorphism in Corollary \ref{cor-composition-product-homology} is obtained here by dimension counting. Is there a more natural construction of this isomorphism? It seems interesting to compare this problem to \cite[Theorem 3.16]{Wu-color-ras}.
\end{question}

One can modify the definition of the $\mathfrak{sl}(N)$ MOY graph polynomial to get a two-variable MOY graph polynomial and use it to define a colored HOMFLY-PT link polynomial. (See for example \cite{Mackaay-Stosic-Vaz2}.) This colored HOMFLY-PT polynomial has been categorified by a colored HOMFLY-PT link homology \cite{Webster-Williamson}. We expect Rasmussen's spectral sequence to generalize to a spectral sequence relating the colored HOMFLY-PT link homology to the colored $\mathfrak{sl}(N)$ link homology. This should imply the following conjecture.

\begin{conjecture}
Denote by $H$ the colored HOMFLY-PT link homology (with an appropriate normalization.) Then
\begin{eqnarray}
\label{eq-colored-sl-N-to-HOMFLY-min} \lim_{N\rightarrow \infty} \frac{\min\deg_q H_N(B^{(m)})}{N-m} & = & \min\deg_x H(B^{(m)}), \\
\label{eq-colored-sl-N-to-HOMFLY-max} \lim_{N\rightarrow \infty} \frac{\max\deg_q H_N(B^{(m)})}{N-m} & = & \max\deg_x H(B^{(m)}),
\end{eqnarray}
where $\deg_x$ is the degree from the $x$-grading which corresponds to the ``framing variable" $x$ of the colored HOMFLY-PT link polynomial.

In particular,
\begin{equation}\label{eq-colored-MFW-HOMFLY}
w -b \leq  \frac{\min\deg_x H(B^{(m)})}{m} \leq  \frac{\max\deg_x H(B^{(m)})}{m} \leq w +b.
\end{equation}
\end{conjecture}

\begin{remark}
It seem possible to prove \eqref{eq-colored-MFW-HOMFLY} without using \eqref{eq-colored-sl-N-to-HOMFLY-min} and \eqref{eq-colored-sl-N-to-HOMFLY-max}. The proof should be a slightly modification of our proof of Theorem \ref{THM-colored-MFW-sl-N-limit}. But this would require a construction of the colored HOMFLY-PT link homology directly modeled on the two-variable MOY graph polynomial.
\end{remark}

For any link $L$, Rutherford \cite{Rutherford} proved that $\min\deg_x \mathsf{P}(L) = w-b$ for some braid representation of $L$ with writhe $w$ and $b$ strands if and only if there exists a Legendrian front projection of $L$ that has an oriented ruling.

\begin{question}
Can one generalize Rutherford's result to a necessary and sufficient condition to the sharpness of any of \eqref{eq-colored-MFW-sl-N-limit-special}, \eqref{eq-colored-MFW-sl-N-limit-general}, \eqref{eq-colored-MFW-sl-N-limit-special-poly} or \eqref{eq-colored-MFW-sl-N-limit-general-poly}?
\end{question}

Applying the Morton-Franks-Williams inequality to cables, Morton and Short \cite{Morton-Short} introduced the cabled Morton-Franks-Williams inequalities. For simplicity, let us only consider knots. Suppose a knot $K$ has a braid diagram of $b$ strands with writhe $w$. Denote by $K_{m,k}$ the $(m,k)$-cable of $K$. Then the braid diagram of $K$ leads to an obvious braid diagram of $K_{m,k}$ of $mb$ strands with writhe $mw+k(m-1)$. Applying \eqref{ineq-MWF-sl-N-KR} to this braid diagram of $K_{m,k}$, one gets

\begin{equation}\label{ineq-MWF-sl-N-KR-cabled}
w-b \leq \liminf_{N\rightarrow \infty}\frac{\min\deg_q H_N(K_{m,k})}{m(N-1)} - \frac{k(m-1)}{m} \leq \limsup_{N\rightarrow \infty}\frac{\max\deg_q H_N(K_{m,k})}{m(N-1)} - \frac{k(m-1)}{m} \leq w+b.
\end{equation}

\begin{question}
How does the colored inequality \eqref{eq-colored-MFW-sl-N-limit-special} compare to the cabled inequality \eqref{ineq-MWF-sl-N-KR-cabled}? Is there an explicit relation between the $\mathfrak{sl}(N)$ homology of $K^{(m)}$ and that of $K_{m,k}$? Will we get better bounds by applying the colored Morton-Franks-Williams inequalities to cables? 
\end{question}

One interesting application of the Morton-Franks-Williams inequality is to verify the following Jones Conjecture for special class of links.

\begin{conjecture}\label{Conj-Jones}\cite[end of Section 8]{Jones-conjecture}
For any link $L$ of braid index $b$, if $B_1$ and $B_2$ are two braid diagrams of $L$ of $b$ strands, then the writhes of $B_1$ and $B_2$ are equal.
\end{conjecture}

If both ends of the Morton-Franks-Williams inequality \eqref{ineq-MFW} are sharp for a closed braid $B$, then half of the $x$-span of the HOMFLY-PT polynomial is equal to the braid index of $B$ and the Jones Conjecture is true for $B$. But, since the Morton-Franks-Williams inequality is in general not sharp, this argument works only for special classes of links. See for example \cite{FW,Kawamuro,Lee-Seo-JC,Murasugi-JC,Nakamura-Jones-conj,Stoimenow-Jones-Conj} for related results.

\begin{question}
It is clear that the sharpness (of both ends) of any of the inequalities \eqref{eq-colored-MFW-sl-N-limit-special}, \eqref{eq-colored-MFW-sl-N-limit-general}, \eqref{eq-colored-MFW-sl-N-limit-special-poly} or \eqref{eq-colored-MFW-sl-N-limit-general-poly} would similarly imply the Jones Conjecture. Can one use these inequalities to obtain further results on the Jones Conjecture?
\end{question}

\subsection{Organization of this paper} The construction of the colored $\mathfrak{sl}(N)$ link homology is used only superficially in the present paper. So no prior experience in the colored $\mathfrak{sl}(N)$ link homology is needed to understand the proofs in this paper. In Section \ref{sec-MOY-graph-homology}, we will review aspects of the colored $\mathfrak{sl}(N)$ link polynomial and homology that are used in our proofs. We will then establish the composition product in Section \ref{sec-composition-product} and apply it to prove the colored Morton-Franks-Williams inequalities in Section \ref{sec-color-MFW-proof}.

\begin{acknowledgments}
I would like to thank Jozef Przytycki for interesting discussions on the history of the composition product.
\end{acknowledgments}

\section{The Colored $\mathfrak{sl}(N)$ Link Polynomial and Homology}\label{sec-MOY-graph-homology}

Using the MOY graph polynomial, Murakami, Ohtsuki and Yamada \cite{MOY} gave an alternative construction of the $\mathfrak{sl}(N)$ Reshetikhin-Turaev polynomial \cite{Resh-Tur1} for links colored by non-negative integers. We now briefly review their construction in \cite{MOY} and the definition of the colored $\mathfrak{sl}(N)$ link homology in \cite{Wu-color}.

\subsection{The MOY graph polynomial}  We review the MOY graph polynomial \cite{MOY} in this subsection. Our notations and normalizations are slightly different from that used in \cite{MOY}.

For a positive integer $N$, define $\Sigma_N= \{2k-N+1|k=0,1,\dots, N-1\}$. Denote by $\mathcal{P}(\Sigma_N)$ the set of subsets of $\Sigma_N$. For a finite set $A$, denote by $\#A$ the cardinality of $A$. Define a function $\pi:\mathcal{P}(\Sigma_N) \times \mathcal{P}(\Sigma_N) \rightarrow \zed_{\geq 0}$ by
\begin{equation}\label{eq-def-pi}
\pi (A_1, A_2) = \# \{(a_1,a_2) \in A_1 \times A_2 ~|~ a_1>a_2\} \text{ for } A_1,~A_2 \in \mathcal{P}(\Sigma_N).
\end{equation}

\begin{figure}[ht]
$
\xymatrix{
\setlength{\unitlength}{1pt}
\begin{picture}(60,40)(-30,0)

\put(0,0){\vector(0,1){20}}

\put(0,20){\vector(1,1){20}}

\put(0,20){\vector(-1,1){20}}

\put(2,5){\tiny{$e$}}

\put(12,26){\tiny{$e_2$}}

\put(-15,26){\tiny{$e_1$}}

\end{picture} & \text{or} & \setlength{\unitlength}{1pt}
\begin{picture}(60,40)(-30,0)

\put(0,20){\vector(0,1){20}}

\put(-20,0){\vector(1,1){20}}

\put(20,0){\vector(-1,1){20}}

\put(2,30){\tiny{$e$}}

\put(12,12){\tiny{$e_2$}}

\put(-15,13){\tiny{$e_1$}}

\end{picture}
} 
$
\caption{}\label{tri-vertex} 

\end{figure}

Let $\Gamma$ be an MOY graph. Denote by $E(\Gamma)$ the set of edges of $\Gamma$, by $V(\Gamma)$ the set of vertices of $\Gamma$ and by $\mathsf{c}:E(\Gamma) \rightarrow \zed_{\geq 0}$ the color function of $\Gamma$. That is, for every edge $e$ of $\Gamma$, $\mathsf{c}(e) \in \zed_{\geq 0}$ is the color of $e$. 

A state of $\Gamma$ is a function $\varphi: E(\Gamma) \rightarrow \mathcal{P}(\Sigma_N)$ such that
\begin{enumerate}[(i)]
	\item for every edge $e$ of $\Gamma$, $\#\varphi(e) = \mathsf{c}(e)$,
	\item for every vertex $v$ of $\Gamma$, as depicted in Figure \ref{tri-vertex}, we have $\varphi(e)=\varphi(e_1) \cup \varphi(e_2)$. 
\end{enumerate}
Note that (i) and (ii) imply that $\varphi(e_1) \cap \varphi(e_2)=\emptyset$.

Denote by $\mathcal{S}_N(\Gamma)$ the set of states of $\Gamma$.

For a state $\varphi$ of $\Gamma$ and a vertex $v$ of $\Gamma$ (as depicted in Figure \ref{tri-vertex}), the weight of $v$ with respect to $\varphi$ is defined to be 
\begin{equation}\label{eq-weight-vertex}
\mathrm{wt}(v;\varphi) = q^{\frac{\mathsf{c}(e_1)\mathsf{c}(e_2)}{2} - \pi(\varphi(e_1),\varphi(e_2))}.
\end{equation}

Given a state $\varphi$ of $\Gamma$, replace each edge $e$ of $\Gamma$ by $\mathsf{c}(e)$ parallel edges, assign to each of these new edges a different element of $\varphi(e)$ and, at every vertex, connect each pair of new edges assigned the same element of $\Sigma_N$. This changes $\Gamma$ into a collection $\mathcal{C}_\varphi$ of embedded circles, each of which is assigned an element of $\Sigma_N$. By abusing notation, we denote by $\varphi(C)$ the element of $\Sigma_N$ assigned to $C\in \mathcal{C}_\varphi$. Note that: 
\begin{itemize}
	\item There may be intersections between different circles in $\mathcal{C}_\varphi$. But, each circle in $\mathcal{C}_\varphi$ is embedded, that is, it has no self-intersections or self-tangencies.
	\item There may be more than one way to do this. But if we view $\mathcal{C}_\varphi$ as a virtue link and the intersection points between different elements of $\mathcal{C}_\varphi$ virtual crossings, then the above construction is unique up to purely virtual regular Reidemeister moves.
\end{itemize}
The rotation number $\mathrm{rot}(\varphi)$ of $\varphi$ is then defined to be
\begin{equation}\label{eq-rot-state}
\mathrm{rot}(\varphi) = \sum_{C\in \mathcal{C}_\varphi} \varphi(C) \mathrm{rot}(C).
\end{equation}
Clearly, $\mathrm{rot}(\varphi)$ is independent of the choices made in its definition. We also make the following simple observation, which will be useful in Section \ref{sec-composition-product}.

\begin{lemma}\label{lemma-total-rot-constant}
For any state $\varphi$ of $\Gamma$, 
\[
\sum_{C\in \mathcal{C}_\varphi} \mathrm{rot}(C) = \mathrm{rot}(\Gamma).
\]
\end{lemma}

Now we are ready to define the $\mathfrak{sl}(N)$ MOY graph polynomial.

\begin{definition}\label{def-MOY-graph-poly}\cite{MOY}
The $\mathfrak{sl}(N)$ MOY graph polynomial of $\Gamma$ is defined to be
\begin{equation}\label{MOY-bracket-def}
\left\langle \Gamma \right\rangle_N := \sum_{\varphi \in \mathcal{S}_N(\Gamma)} (\prod_{v \in V(\Gamma)} \mathrm{wt}(v;\varphi)) q^{\mathrm{rot}(\varphi)}.
\end{equation}
\end{definition}

\begin{remark} 
\begin{enumerate}[1.]
	\item If $\Gamma$ contains an edge with color greater than $N$, then $\mathcal{S}_N(\Gamma) = \emptyset$ and therefore, $\left\langle \Gamma \right\rangle_N=0$.
	\item If $\Gamma$ contains an edge with color $0$, then erasing this edge does not change the polynomial $\left\langle \Gamma \right\rangle_N$.
	\item We allow $\Gamma$ to be the empty graph and use the convention $\left\langle \emptyset \right\rangle_N=1$.
\end{enumerate}
\end{remark}

\subsection{The colored $\mathfrak{sl}(N)$ link polynomial} The $\mathfrak{sl}(N)$ Reshetikhin-Turaev polynomial \cite{Resh-Tur1} for links colored by non-negative integers can be expressed as a combination of the $\mathfrak{sl}(N)$ MOY graph polynomials of the MOY resolutions of its diagram. 

\begin{figure}[ht]
$
\xymatrix{
\setlength{\unitlength}{1pt}
\begin{picture}(40,40)(-20,-20)

\put(-20,-20){\vector(1,1){40}}

\put(20,-20){\line(-1,1){15}}

\put(-5,5){\vector(-1,1){15}}

\put(-11,15){\tiny{$_m$}}

\put(9,15){\tiny{$_n$}}

\end{picture} & \text{or} & \setlength{\unitlength}{1pt}
\begin{picture}(40,40)(-20,-20)

\put(20,-20){\vector(-1,1){40}}

\put(-20,-20){\line(1,1){15}}

\put(5,5){\vector(1,1){15}}

\put(-11,15){\tiny{$_m$}}

\put(9,15){\tiny{$_n$}}

\end{picture}
} 
$
\caption{}\label{crossings-fig} 

\end{figure}

Let $D$ be a diagram of a link colored by non-negative integers. An MOY resolution of $D$ is an MOY graph obtained by replacing each crossing of $D$ (as shown in Figure \ref{crossings-fig}) by the shape in Figure \ref{MOY-res-fig} for some integer $k$ satisfying $\max\{0,m-n\}\leq k \leq m$. Denote by $\mathcal{R}(D)$ the set of all MOY resolutions of $D$.

\begin{figure}[ht]
$
\xymatrix{
\setlength{\unitlength}{1pt}
\begin{picture}(70,60)(-35,0)

\put(-15,0){\vector(0,1){20}}
\put(-15,20){\vector(0,1){20}}
\put(-15,40){\vector(0,1){20}}
\put(15,0){\vector(0,1){20}}
\put(15,20){\vector(0,1){20}}
\put(15,40){\vector(0,1){20}}

\put(15,20){\vector(-1,0){30}}
\put(-15,40){\vector(1,0){30}}

\put(-25,5){\tiny{$_{n}$}}
\put(-25,55){\tiny{$_{m}$}}
\put(-30,30){\tiny{$_{n+k}$}}

\put(-2,15){\tiny{$_{k}$}}
\put(-12,43){\tiny{$_{n+k-m}$}}

\put(18,5){\tiny{$_{m}$}}
\put(18,55){\tiny{$_{n}$}}
\put(18,30){\tiny{$_{m-k}$}}

\end{picture}
} 
$
\caption{}\label{MOY-res-fig} 

\end{figure}

\begin{definition}\cite{MOY}\label{sl-N-poly-def}
For a link diagram $D$ colored by non-negative integers, define the unnormalized Reshetikhin-Turaev $\mathfrak{sl}(N)$ polynomial $\left\langle D \right\rangle_N$ of $D$ by applying the following skein sum at every crossing of $D$.
\[
\left\langle \setlength{\unitlength}{1pt}
\begin{picture}(40,40)(-20,0)

\put(-20,-20){\vector(1,1){40}}

\put(20,-20){\line(-1,1){15}}

\put(-5,5){\vector(-1,1){15}}

\put(-11,15){\tiny{$_m$}}

\put(9,15){\tiny{$_n$}}

\end{picture} \right\rangle_N = \sum_{k=\max\{0,m-n\}}^{m} (-1)^{m-k} q^{k-m}\left\langle \setlength{\unitlength}{1pt}
\begin{picture}(70,60)(-35,30)

\put(-15,0){\vector(0,1){20}}
\put(-15,20){\vector(0,1){20}}
\put(-15,40){\vector(0,1){20}}
\put(15,0){\vector(0,1){20}}
\put(15,20){\vector(0,1){20}}
\put(15,40){\vector(0,1){20}}

\put(15,20){\vector(-1,0){30}}
\put(-15,40){\vector(1,0){30}}

\put(-25,5){\tiny{$_{n}$}}
\put(-25,55){\tiny{$_{m}$}}
\put(-30,30){\tiny{$_{n+k}$}}

\put(-2,15){\tiny{$_{k}$}}
\put(-12,43){\tiny{$_{n+k-m}$}}

\put(18,5){\tiny{$_{m}$}}
\put(18,55){\tiny{$_{n}$}}
\put(18,30){\tiny{$_{m-k}$}}

\end{picture}\right\rangle_N
\]
\[
\left\langle \setlength{\unitlength}{1pt}
\begin{picture}(40,40)(-20,0)

\put(20,-20){\vector(-1,1){40}}

\put(-20,-20){\line(1,1){15}}

\put(5,5){\vector(1,1){15}}

\put(-11,15){\tiny{$_m$}}

\put(9,15){\tiny{$_n$}}

\end{picture} \right\rangle_N = \sum_{k=\max\{0,m-n\}}^{m} (-1)^{k-m} q^{m-k}\left\langle \setlength{\unitlength}{1pt}
\begin{picture}(70,60)(-35,30)

\put(-15,0){\vector(0,1){20}}
\put(-15,20){\vector(0,1){20}}
\put(-15,40){\vector(0,1){20}}
\put(15,0){\vector(0,1){20}}
\put(15,20){\vector(0,1){20}}
\put(15,40){\vector(0,1){20}}

\put(15,20){\vector(-1,0){30}}
\put(-15,40){\vector(1,0){30}}

\put(-25,5){\tiny{$_{n}$}}
\put(-25,55){\tiny{$_{m}$}}
\put(-30,30){\tiny{$_{n+k}$}}

\put(-2,15){\tiny{$_{k}$}}
\put(-12,43){\tiny{$_{n+k-m}$}}

\put(18,5){\tiny{$_{m}$}}
\put(18,55){\tiny{$_{n}$}}
\put(18,30){\tiny{$_{m-k}$}}

\end{picture}\right\rangle_N
\]
\vspace{.5cm}

Also, for each crossing $c$ of $D$, define the shifting factor $\mathsf{s}(c)$ of $c$ by
\[
\mathsf{s}\left(\setlength{\unitlength}{1pt}
\begin{picture}(40,40)(-20,0)

\put(-20,-20){\vector(1,1){40}}

\put(20,-20){\line(-1,1){15}}

\put(-5,5){\vector(-1,1){15}}

\put(-11,15){\tiny{$_m$}}

\put(9,15){\tiny{$_n$}}

\end{picture}\right) = 
\begin{cases}
(-1)^{-m} q^{m(N+1-m)} & \text{if } m=n,\\
1 & \text{if } m \neq n,
\end{cases}
\]
\[
\mathsf{s}\left(\setlength{\unitlength}{1pt}
\begin{picture}(40,40)(-20,0)

\put(20,-20){\vector(-1,1){40}}

\put(-20,-20){\line(1,1){15}}

\put(5,5){\vector(1,1){15}}

\put(-11,15){\tiny{$_m$}}

\put(9,15){\tiny{$_n$}}

\end{picture}\right) = 
\begin{cases}
(-1)^m q^{-m(N+1-m)} & \text{if } m=n,\\
1 & \text{if } m \neq n.
\end{cases}
\]

The normalized Reshetikhin-Turaev $\mathfrak{sl}(N)$ polynomial $\mathsf{P}_N(D)$ of $D$ is defined to be
\[
\mathsf{P}_N(D) = \left\langle D \right\rangle_N \cdot \prod_c \mathsf{s}(c),
\]
where $c$ runs through all crossings of $D$.
\end{definition}

\begin{theorem}\cite{MOY}
$\left\langle D \right\rangle_N$ is invariant under Reidemeister moves (II) and (III). $\mathsf{P}_N(D)$ is invariant under all Reidemeister moves.
\end{theorem}

\subsection{The colored $\mathfrak{sl}(N)$ link homology} For a positive integer $N$ and an MOY graph $\Gamma$, I defined in \cite{Wu-color} a $\zed$-graded homology $H_N(\Gamma)$ whose graded dimension is $\left\langle \Gamma \right\rangle_N$. We call this grading the quantum grading of $H_N(\Gamma)$. 

Next we give $H_N(\Gamma)$ a homological grading such that the whole of $H_N(\Gamma)$ has homological grading $0$. This makes it a $\zed^{\oplus2}$-graded space.

For the resolution of a crossing, define
\begin{equation}\label{eq-h-grading-def-+}
\mathsf{s}_{h,N}\left(\setlength{\unitlength}{1pt}
\begin{picture}(40,40)(-20,0)

\put(-20,-20){\vector(1,1){40}}

\put(20,-20){\line(-1,1){15}}

\put(-5,5){\vector(-1,1){15}}

\put(-11,15){\tiny{$_m$}}

\put(9,15){\tiny{$_n$}}

\end{picture};\setlength{\unitlength}{1pt}
\begin{picture}(70,60)(-35,30)

\put(-15,0){\vector(0,1){20}}
\put(-15,20){\vector(0,1){20}}
\put(-15,40){\vector(0,1){20}}
\put(15,0){\vector(0,1){20}}
\put(15,20){\vector(0,1){20}}
\put(15,40){\vector(0,1){20}}

\put(15,20){\vector(-1,0){30}}
\put(-15,40){\vector(1,0){30}}

\put(-25,5){\tiny{$_{n}$}}
\put(-25,55){\tiny{$_{m}$}}
\put(-30,30){\tiny{$_{n+k}$}}

\put(-2,15){\tiny{$_{k}$}}
\put(-12,43){\tiny{$_{n+k-m}$}}

\put(18,5){\tiny{$_{m}$}}
\put(18,55){\tiny{$_{n}$}}
\put(18,30){\tiny{$_{m-k}$}}

\end{picture}\right) = \begin{cases}
-k & \text{if } m=n,\\
m-k & \text{if } m\neq n,
\end{cases}
\end{equation}
\begin{equation}\label{eq-h-grading-def--}
\mathsf{s}_{h,N}\left(\setlength{\unitlength}{1pt}
\begin{picture}(40,40)(-20,0)

\put(20,-20){\vector(-1,1){40}}

\put(-20,-20){\line(1,1){15}}

\put(5,5){\vector(1,1){15}}

\put(-11,15){\tiny{$_m$}}

\put(9,15){\tiny{$_n$}}

\end{picture};\setlength{\unitlength}{1pt}
\begin{picture}(70,60)(-35,30)

\put(-15,0){\vector(0,1){20}}
\put(-15,20){\vector(0,1){20}}
\put(-15,40){\vector(0,1){20}}
\put(15,0){\vector(0,1){20}}
\put(15,20){\vector(0,1){20}}
\put(15,40){\vector(0,1){20}}

\put(15,20){\vector(-1,0){30}}
\put(-15,40){\vector(1,0){30}}

\put(-25,5){\tiny{$_{n}$}}
\put(-25,55){\tiny{$_{m}$}}
\put(-30,30){\tiny{$_{n+k}$}}

\put(-2,15){\tiny{$_{k}$}}
\put(-12,43){\tiny{$_{n+k-m}$}}

\put(18,5){\tiny{$_{m}$}}
\put(18,55){\tiny{$_{n}$}}
\put(18,30){\tiny{$_{m-k}$}}

\end{picture}\right) = \begin{cases}
k & \text{if } m=n,\\
k-m & \text{if } m\neq n,
\end{cases}
\end{equation}
\begin{equation}\label{eq-q-grading-def-+}
\mathsf{s}_{q,N}\left(\setlength{\unitlength}{1pt}
\begin{picture}(40,40)(-20,0)

\put(-20,-20){\vector(1,1){40}}

\put(20,-20){\line(-1,1){15}}

\put(-5,5){\vector(-1,1){15}}

\put(-11,15){\tiny{$_m$}}

\put(9,15){\tiny{$_n$}}

\end{picture};\setlength{\unitlength}{1pt}
\begin{picture}(70,60)(-35,30)

\put(-15,0){\vector(0,1){20}}
\put(-15,20){\vector(0,1){20}}
\put(-15,40){\vector(0,1){20}}
\put(15,0){\vector(0,1){20}}
\put(15,20){\vector(0,1){20}}
\put(15,40){\vector(0,1){20}}

\put(15,20){\vector(-1,0){30}}
\put(-15,40){\vector(1,0){30}}

\put(-25,5){\tiny{$_{n}$}}
\put(-25,55){\tiny{$_{m}$}}
\put(-30,30){\tiny{$_{n+k}$}}

\put(-2,15){\tiny{$_{k}$}}
\put(-12,43){\tiny{$_{n+k-m}$}}

\put(18,5){\tiny{$_{m}$}}
\put(18,55){\tiny{$_{n}$}}
\put(18,30){\tiny{$_{m-k}$}}

\end{picture}\right) = \begin{cases}
k-m+m(N+1-m) & \text{if } m=n,\\
k-m & \text{if } m\neq n,
\end{cases}
\end{equation}
\begin{equation}\label{eq-q-grading-def--}
\mathsf{s}_{q,N}\left(\setlength{\unitlength}{1pt}
\begin{picture}(40,40)(-20,0)

\put(20,-20){\vector(-1,1){40}}

\put(-20,-20){\line(1,1){15}}

\put(5,5){\vector(1,1){15}}

\put(-11,15){\tiny{$_m$}}

\put(9,15){\tiny{$_n$}}

\end{picture};\setlength{\unitlength}{1pt}
\begin{picture}(70,60)(-35,30)

\put(-15,0){\vector(0,1){20}}
\put(-15,20){\vector(0,1){20}}
\put(-15,40){\vector(0,1){20}}
\put(15,0){\vector(0,1){20}}
\put(15,20){\vector(0,1){20}}
\put(15,40){\vector(0,1){20}}

\put(15,20){\vector(-1,0){30}}
\put(-15,40){\vector(1,0){30}}

\put(-25,5){\tiny{$_{n}$}}
\put(-25,55){\tiny{$_{m}$}}
\put(-30,30){\tiny{$_{n+k}$}}

\put(-2,15){\tiny{$_{k}$}}
\put(-12,43){\tiny{$_{n+k-m}$}}

\put(18,5){\tiny{$_{m}$}}
\put(18,55){\tiny{$_{n}$}}
\put(18,30){\tiny{$_{m-k}$}}

\end{picture}\right) = \begin{cases}
m-k-m(N+1-m) & \text{if } m=n,\\
m-k & \text{if } m\neq n.
\end{cases}
\end{equation}

For a link diagram $D$ colored by non-negative integers and an MOY resolution $\Gamma$ of $D$, define $\mathsf{s}_{h,N}(D;\Gamma)$ to be the sum of the values of $\mathsf{s}_{h,N}$ over all crossings of $D$ and define $\mathsf{s}_{q,N}(D;\Gamma)$ to be the sum of the values of $\mathsf{s}_{q,N}$ over all crossings of $D$. 

Denote by $H_N(\Gamma)\|\mathsf{s}_{h,N}(D;\Gamma)\|\{q^{\mathsf{s}_{q,N}(D;\Gamma)}\}$ the space obtained from $H_N(\Gamma)$ by shifting its homological grading by $\mathsf{s}_{h,N}(D;\Gamma)$ and shifting its quantum grading by $\mathsf{s}_{q,N}(D;\Gamma)$.

\begin{theorem}\label{thm-sl-N-homology}\cite{Wu-color}
Let $D$ be a link diagram colored by non-negative integers. Then one can equip the $\zed^{\oplus2}$-graded space 
\[
\bigoplus_{\Gamma \in \mathcal{R}(D)} H_N(\Gamma)\|\mathsf{s}_{h,N}(D;\Gamma)\|\{q^{\mathsf{s}_{q,N}(D;\Gamma)}\}
\]
with a homogeneous differential map of quantum grading $0$ and homological grading $1$ so that the homology of this chain complex, with its $\zed^{\oplus2}$-grading, is invariant under all Reidemeister moves.

We denote this invariant link homology by $H_N$. From the form of the above chain complex, one can see that the graded Euler characteristic of $H_N(D)$ is the normalized Reshetikhin-Turaev $\mathfrak{sl}(N)$ polynomial $\mathsf{P}_N(D)$.
\end{theorem}

\section{Proof of The Composition Product}\label{sec-composition-product}

Jaeger \cite{Jaeger-composition} proved his composition product formula by showing that the composition product satisfies the skein relation that uniquely characterizes the HOMFLY-PT polynomial. Similarly, Wagner \cite{Wagner-composition} proved his composition product formula by showing that the composition product satisfies the MOY relations that uniquely characterizes the $1,2$-colored MOY graph polynomial. The proof of Theorem \ref{THM-composition-product} would be rather lengthy if we use a direct generalization of their approach. Fortunately, the composition product \eqref{eq-composition-product} in Theorem \ref{THM-composition-product} is a simple corollary of the MOY state sum formula \eqref{MOY-bracket-def}, which makes the proof a lot easier. In fact, it is not hard to see that \eqref{eq-composition-product} and \eqref{MOY-bracket-def} are actually equivalent to each other.

\begin{definition}\label{def-MOY-state-split}
Let $\Gamma$ be an MOY graph, and $M$, $N$ two positive integers. For a state $\varphi \in \mathcal{S}_{M+N}(\Gamma)$, define $\varphi_1:E(\Gamma) \rightarrow \mathcal{P}(\Sigma_M)$ by 
\[
\varphi_1(e) = \{k~|~k-N \in \varphi(e)\cap \{2k-M-N+1~|~k=0,1,\dots,M-1\}\}
\]
and $\varphi_2:E(\Gamma) \rightarrow \mathcal{P}(\Sigma_N)$ by 
\[
\varphi_2(e) = \{k~|~k+M \in \varphi(e)\cap \{2k-M-N+1~|~k=M,M+1,\dots,M+N-1\}\}.
\]
Moreover, define $\mathsf{f}_\varphi:E(\Gamma)\rightarrow \zed_{\geq 0}$ by $\mathsf{f}_\varphi(e)=\#\varphi_1(e)$. 

It is easy to see that $\mathsf{f}_\varphi \in \mathcal{L}(\Gamma)$, $\varphi_1 \in \mathcal{S}_{M}(\Gamma_{\mathsf{f}_\varphi})$ and $\varphi_2 \in \mathcal{S}_{N}(\Gamma_{\bar{\mathsf{f}}_\varphi})$.

For $\mathsf{f} \in \mathcal{L}(\Gamma)$, define $\mathcal{S}_{M+N}^{\mathsf{f}}(\Gamma) = \{\varphi \in \mathcal{S}_{M+N}(\Gamma)~|~ \mathsf{f}_\varphi = \mathsf{f}\}$.
\end{definition}

\begin{lemma}\label{lemma-MOY-state-split-decomp-1}
\[
\mathcal{S}_{M+N}(\Gamma) = \bigsqcup_{\mathsf{f} \in \mathcal{L}(\Gamma)} \mathcal{S}_{M+N}^{\mathsf{f}}(\Gamma)
\]
and therefore
\begin{equation}\label{eq-MOY-state-sum-split}
\left\langle \Gamma \right\rangle_{M+N} = \sum_{\mathsf{f} \in \mathcal{L}(\Gamma)} \sum_{\varphi \in \mathcal{S}_{M+N}^\mathsf{f}(\Gamma)} (\prod_{v \in V(\Gamma)} \mathrm{wt}(v;\varphi)) q^{\mathrm{rot}(\varphi)}.
\end{equation}
\end{lemma}

\begin{proof}
This lemma follows easily from the relevant definitions. We leave the details to the reader.
\end{proof}

\begin{lemma}\label{lemma-MOY-state-split-decomp-2}
For any $\mathsf{f} \in \mathcal{L}(\Gamma)$, the function $\mathcal{S}_{M+N}^{\mathsf{f}}(\Gamma) \rightarrow \mathcal{S}_{M}(\Gamma_{\mathsf{f}}) \times \mathcal{S}_{N}(\Gamma_{\bar{\mathsf{f}}})$ given by $\varphi \mapsto (\varphi_1,\varphi_2)$ is a bijection. Moreover, for any $\varphi \in \mathcal{S}_{M+N}^\mathsf{f}(\Gamma)$, 
\begin{eqnarray}
\label{eq-MOY-split-rot} \mathrm{rot}(\varphi) & = & \mathrm{rot}(\varphi_1) + \mathrm{rot}(\varphi_2) -N \cdot \mathrm{rot}(\Gamma_{\mathsf{f}}) +M \cdot \mathrm{rot}(\Gamma_{\bar{\mathsf{f}}}), \\
\label{eq-MOY-split-wt} \mathrm{wt}(v;\varphi) & = & \mathrm{wt}(v;\varphi_1)\cdot \mathrm{wt}(v;\varphi_2)\cdot q^{[v|\Gamma|\mathsf{f}]},
\end{eqnarray}
where $v$ is any vertex of $\Gamma$.
\end{lemma}

\begin{proof}
By the definition of $\varphi_1$ and $\varphi_2$, it is clear that $\varphi \mapsto (\varphi_1,\varphi_2)$ gives a bijection $\mathcal{S}_{M+N}^{\mathsf{f}}(\Gamma) \rightarrow \mathcal{S}_{M}(\Gamma_{\mathsf{f}}) \times \mathcal{S}_{N}(\Gamma_{\bar{\mathsf{f}}})$. Equation \eqref{eq-MOY-split-rot} follows easily from Lemma \ref{lemma-total-rot-constant} and the definitions of rotation numbers (equations \eqref{eq-def-rot-MOY} and \eqref{eq-rot-state}.) It remains to prove \eqref{eq-MOY-split-wt}.

Denote by $\mathsf{c}$ the color function of $\Gamma$. Then $\mathsf{c}= \mathsf{f} + \bar{\mathsf{f}}$. Let $v$ be a vertex of $\Gamma$ as shown in Figure \ref{fig-MOY-vertex}. Recall that
\begin{eqnarray*}
\mathrm{wt}(v;\varphi) & = & q^{\frac{\mathsf{c}(e_1)\mathsf{c}(e_2)}{2} - \pi(\varphi(e_1),\varphi(e_2))}, \\
\mathrm{wt}(v;\varphi_1) & = & q^{\frac{\mathsf{f}(e_1)\mathsf{f}(e_2)}{2} - \pi(\varphi_1(e_1),\varphi_1(e_2))}, \\
\mathrm{wt}(v;\varphi_2) & = & q^{\frac{\bar{\mathsf{f}}(e_1)\bar{\mathsf{f}}(e_2)}{2} - \pi(\varphi_2(e_1),\varphi_2(e_2))},
\end{eqnarray*}
where $\pi$ is defined in \eqref{eq-def-pi}. Let 
\begin{eqnarray*}
\Sigma' & = & \{2k-M-N+1~|~k=0,1,\dots,M-1\}, \\
\Sigma'' & = & \{2k-M-N+1~|~k=M,M+1,\dots,M+N-1\}.
\end{eqnarray*}
Then $\Sigma_{M+N}= \Sigma' \sqcup \Sigma''$ and 
\begin{eqnarray*}
&& \pi(\varphi(e_1),\varphi(e_2)) \\
& = & \pi((\varphi(e_1)\cap \Sigma') \sqcup (\varphi(e_1)\cap \Sigma''),(\varphi(e_2)\cap \Sigma') \sqcup (\varphi(e_2)\cap \Sigma'')) \\
& = & \pi(\varphi(e_1)\cap \Sigma',\varphi(e_2)\cap \Sigma') + \pi(\varphi(e_1)\cap \Sigma'',\varphi(e_2)\cap \Sigma'') \\
& & + ~\pi(\varphi(e_1)\cap \Sigma'',\varphi(e_2)\cap \Sigma') + \pi(\varphi(e_1)\cap \Sigma',\varphi(e_2)\cap \Sigma'') \\
& = & \pi(\varphi_1(e_1),\varphi_1(e_2)) + \pi(\varphi_2(e_1),\varphi_2(e_2)) + \bar{\mathsf{f}}(e_1)\mathsf{f}(e_2).
\end{eqnarray*}
Thus,
\begin{eqnarray*}
&& (\frac{\mathsf{c}(e_1)\mathsf{c}(e_2)}{2} - \pi(\varphi(e_1),\varphi(e_2))) - (\frac{\mathsf{f}(e_1)\mathsf{f}(e_2)}{2} - \pi(\varphi_1(e_1),\varphi_1(e_2))) \\
&& - ~(\frac{\bar{\mathsf{f}}(e_1)\bar{\mathsf{f}}(e_2)}{2} - \pi(\varphi_2(e_1),\varphi_2(e_2))) \\
& = & \frac{(\mathsf{f}(e_1) + \bar{\mathsf{f}}(e_1))(\mathsf{f}(e_2) + \bar{\mathsf{f}}(e_2)) - \mathsf{f}(e_1)\mathsf{f}(e_2) - \bar{\mathsf{f}}(e_1)\bar{\mathsf{f}}(e_2)}{2} - \bar{\mathsf{f}}(e_1)\mathsf{f}(e_2) \\
& = & \frac{1}{2} (\mathsf{f}(e_1)\bar{\mathsf{f}}(e_2) - \bar{\mathsf{f}}(e_1)\mathsf{f}(e_2)) = [v|\Gamma|\mathsf{f}].
\end{eqnarray*}
This proves \eqref{eq-MOY-split-wt}.
\end{proof}

Theorem \ref{THM-composition-product} follows easily from Lemmas \ref{lemma-MOY-state-split-decomp-1} and \ref{lemma-MOY-state-split-decomp-2}.

\begin{proof}[Proof of Theorem \ref{THM-composition-product}]
By Lemmas \ref{lemma-MOY-state-split-decomp-1} and \ref{lemma-MOY-state-split-decomp-2},
\begin{eqnarray*}
&& \left\langle \Gamma \right\rangle_{M+N} \\
& = & \sum_{\mathsf{f} \in \mathcal{L}(\Gamma)} \sum_{\varphi \in \mathcal{S}_{M+N}^\mathsf{f}(\Gamma)} (\prod_{v \in V(\Gamma)} \mathrm{wt}(v;\varphi)) q^{\mathrm{rot}(\varphi)} \\
& = & \sum_{\mathsf{f} \in \mathcal{L}(\Gamma)} \sum_{\varphi \in \mathcal{S}_{M+N}^\mathsf{f}(\Gamma)} (\prod_{v \in V(\Gamma)} \mathrm{wt}(v;\varphi_1)\cdot \mathrm{wt}(v;\varphi_2)\cdot q^{[v|\Gamma|\mathsf{f}]}) q^{\mathrm{rot}(\varphi_1) + \mathrm{rot}(\varphi_2) -N \cdot \mathrm{rot}(\Gamma_{\mathsf{f}}) +M \cdot \mathrm{rot}(\Gamma_{\bar{\mathsf{f}}})} \\
& = & \sum_{\mathsf{f} \in \mathcal{L}(\Gamma)} q^{\sigma_{M,N}(\Gamma,\mathsf{f})} \sum_{(\varphi_1,\varphi_2) \in \mathcal{S}_{M}(\Gamma_{\mathsf{f}}) \times \mathcal{S}_{N}(\Gamma_{\bar{\mathsf{f}}})} (\prod_{v \in V(\Gamma)} \mathrm{wt}(v;\varphi_1)\cdot q^{\mathrm{rot}(\varphi_1)}) \cdot (\prod_{v \in V(\Gamma)} \mathrm{wt}(v;\varphi_2)\cdot q^{\mathrm{rot}(\varphi_2)}) \\
& = & \sum_{\mathsf{f} \in \mathcal{L}(\Gamma)} q^{\sigma_{M,N}(\Gamma,\mathsf{f})} \cdot \left\langle \Gamma_{\mathsf{f}} \right\rangle_M \cdot \left\langle \Gamma_{\bar{\mathsf{f}}} \right\rangle_N.
\end{eqnarray*}
\end{proof}

\begin{remark}
The above proof shows that the MOY state sum formula \eqref{MOY-bracket-def} implies the composition product \eqref{eq-composition-product}. Now consider $\left\langle \Gamma \right\rangle_{N}$ as $\left\langle \Gamma \right\rangle_{\underbrace{1+\cdots+1}_{N ~1's}}$ and use the composition product \eqref{eq-composition-product} repeatedly. This give a state sum formula of $\left\langle \Gamma \right\rangle_{N}$ in terms of $\left\langle \right\rangle_{1}$ of a family of simple MOY graphs, each of which is a collection of embedded circles colored by $1$. (A special case of the formula is given in \cite{Wagner-composition}.) It is not very hard to see that this state sum formula is exactly the MOY state sum formula \eqref{MOY-bracket-def}. Thus \eqref{eq-composition-product} and \eqref{MOY-bracket-def} are equivalent.
\end{remark}

\section{Proof of the Colored Morton-Franks-Williams Inequalities}\label{sec-color-MFW-proof}

We prove in this section Proposition \ref{Prop-colored-MFW-sl-N}, which implies the colored Morton-Franks-Williams inequalities \eqref{eq-colored-MFW-sl-N-limit-special}, \eqref{eq-colored-MFW-sl-N-limit-general}, \eqref{eq-colored-MFW-sl-N-limit-special-poly} or \eqref{eq-colored-MFW-sl-N-limit-general-poly}. We do so by establish upper and lower bounds for the degree of the $\mathfrak{sl}(N)$ MOY graph polynomial of a special type of MOY graphs, which we call MOY tracks.

\subsection{MOY tracks}

\begin{definition}\label{def-MOY-tracks}
Let $b$ be a positive integer. Assume $0\leq x_1<\cdots<x_b \leq 1$. An MOY track of $b$ strands is an MOY graph $\Gamma$ satisfying:
\begin{enumerate}[1.]
	\item The part of $\Gamma$ outside $[0,1]\times[0,1]$ consists of $b$ edges such that, for each $i=1,\dots,b$, one of these $b$ edges connects $(x_i,0)$ to $(x_i,1)$ via the right side of $[0,1]\times[0,1]$.
	\item The part of $\Gamma$ inside $[0,1]\times[0,1]$ consists of vertical and horizontal edges only.
	\item All vertical edges of $\Gamma$ inside $[0,1]\times[0,1]$ point downward.
	\item The union of all vertical edges of $\Gamma$ inside $[0,1]\times[0,1]$, as a point set, is the set $\{x_1,\dots,x_b\}\times [0,1]$.
	\item Each horizontal edge inside $[0,1]\times[0,1]$ starts and ends on adjacent vertical edges inside $[0,1]\times[0,1]$.
\end{enumerate}
\end{definition}

Note that the MOY resolutions of a braid diagram colored by non-negative integers are all MOY tracks.

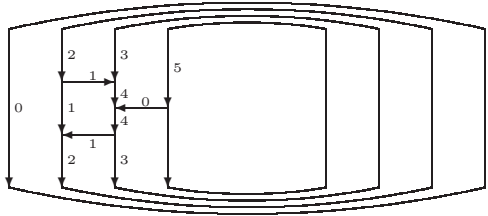
\begin{figure}[ht]
\[
\xymatrix{
\setlength{\unitlength}{1pt}
\begin{picture}(180,100)(0,-20)


\put(0,60){\vector(0,-1){60}}

\put(20,60) {\vector(0,-1){20}}

\put(20,40) {\vector(0,-1){20}}

\put(20,20) {\vector(0,-1){20}}

\put(40,60) {\vector(0,-1){20}}

\put(40,40) {\vector(0,-1){10}}

\put(40,30) {\vector(0,-1){10}}

\put(40,20) {\vector(0,-1){20}}

\put(60,60) {\vector(0,-1){30}}

\put(60,30) {\vector(0,-1){30}}


\put(20,40) {\vector(1,0){20}}

\put(40,20) {\vector(-1,0){20}}

\put(60,30) {\vector(-1,0){20}}


\put(120,0) {\line(0,1){60}}

\put(140,0) {\line(0,1){60}}

\put(160,0) {\line(0,1){60}}

\put(180,0) {\line(0,1){60}}

\qbezier(60,0)(90,-5)(120,0)

\qbezier(40,0)(90,-10)(140,0)

\qbezier(20,0)(90,-15)(160,0)

\qbezier(0,0)(90,-20)(180,0)

\qbezier(60,60)(90,65)(120,60)

\qbezier(40,60)(90,70)(140,60)

\qbezier(20,60)(90,75)(160,60)

\qbezier(0,60)(90,80)(180,60)


\put(2,30){\tiny{$_{0}$}}

\put(22,50){\tiny{$_{2}$}}

\put(22,30){\tiny{$_{1}$}}

\put(22,10){\tiny{$_{2}$}}

\put(30,42){\tiny{$_{1}$}}

\put(30,16){\tiny{$_{1}$}}

\put(42,50){\tiny{$_{3}$}}

\put(42,10){\tiny{$_{3}$}}

\put(42,35){\tiny{$_{4}$}}

\put(42,25){\tiny{$_{4}$}}

\put(50,32){\tiny{$_{0}$}}

\put(62,45){\tiny{$_{5}$}}

\end{picture}
}
\]
\caption{An MOY track of $4$ strands}\label{fig-MOY-track}
\end{figure}

The goal of this subsection is to prove Proposition \ref{prop-bounds-MOY-tracks}, which gives us upper and lower bounds for the degree of the $\mathfrak{sl}(N)$ MOY graph polynomial of an MOY track. For simplicity, we introduce the following notations.

\begin{definition}\label{def-rho}
Let $\Gamma$ be an MOY graph and $v$ a vertex of $\Gamma$ as shown in Figure \ref{fig-MOY-vertex}. Define $\rho_{\Gamma}(v)=\frac{mn}{2}$ and 
\[
\rho(\Gamma) = \sum_{v\in V(\Gamma)} \rho_{\Gamma}(v).
\]
\end{definition}

\begin{proposition}\label{prop-bounds-MOY-tracks}
Let $\Gamma$ be an MOY track of $b$ strands colored by $0,1,\dots,N$. Then 
\begin{equation}\label{eq-bounds-MOY-tracks}
-\mathrm{rot}(\Gamma)(N-\frac{\mathrm{rot}(\Gamma)}{b})-\rho(\Gamma) \leq \min\deg_q \left\langle \Gamma \right\rangle_N \leq \max\deg_q \left\langle \Gamma \right\rangle_N \leq \mathrm{rot}(\Gamma)(N-\frac{\mathrm{rot}(\Gamma)}{b}) + \rho(\Gamma).
\end{equation}
\end{proposition}

\begin{proof}
Note that $\rho(\Gamma) \geq 0$. Also, since $\frac{\mathrm{rot}(\Gamma)}{b}$ is the average of the colors of of the $b$ edges of $\Gamma$ outside $[0,1]\times[0,1]$, we have $0 \leq \frac{\mathrm{rot}(\Gamma)}{b} \leq N$.

We prove inequality \eqref{eq-bounds-MOY-tracks} by an induction on $N$. If $N=1$, then $\Gamma$ is colored by $0,1$. It is then easy to check that $\rho(\Gamma)=0$, $\left\langle \Gamma \right\rangle_N=1$ and therefore, $\min\deg_q \left\langle \Gamma \right\rangle_N = \max\deg_q \left\langle \Gamma \right\rangle_N=0$. This implies that \eqref{eq-bounds-MOY-tracks} holds for $N=1$. 

Now assume \eqref{eq-bounds-MOY-tracks} is true for $N$. Suppose $\Gamma$ is an MOY track colored by $0,1,\dots,N+1$. Using the composition product (Theorem \ref{THM-composition-product},) we get
\begin{equation}\label{eq-proof-bounds-MOY-tracks-1}
\left\langle \Gamma \right\rangle_{N+1} = \sum_{\mathsf{f} \in \mathcal{L}(\Gamma)} q^{\sigma_{N,1}(\Gamma,\mathsf{f})} \cdot \left\langle \Gamma_{\mathsf{f}} \right\rangle_N \cdot \left\langle \Gamma_{\bar{\mathsf{f}}} \right\rangle_1.
\end{equation}
Clearly, for a term $q^{\sigma_{N,1}(\Gamma,\mathsf{f})} \cdot \left\langle \Gamma_{\mathsf{f}} \right\rangle_N \cdot \left\langle \Gamma_{\bar{\mathsf{f}}} \right\rangle_1$ in \eqref{eq-proof-bounds-MOY-tracks-1} to be non-zero, $\mathsf{f}$ must satisfy that $0 \leq \mathsf{f}(e) \leq N$ and $0 \leq \bar{\mathsf{f}}(e) \leq 1$ for all $e \in E(\Gamma)$, which implies that $0 \leq \frac{\mathrm{rot}(\Gamma_{\mathsf{f}})}{b} \leq N$ and $0 \leq \frac{\mathrm{rot}(\Gamma_{\bar{\mathsf{f}}})}{b} \leq 1$. Moreover, by the induction hypothesis, \eqref{eq-bounds-MOY-tracks} is true for $\left\langle \Gamma_{\mathsf{f}} \right\rangle_N$ and $\left\langle \Gamma_{\bar{\mathsf{f}}} \right\rangle_1$.

By direct computation, we get 
\begin{eqnarray*}
&& -\mathrm{rot}(\Gamma)(N+1-\frac{\mathrm{rot}(\Gamma)}{b}) + \mathrm{rot}(\Gamma_{\mathsf{f}})(N-\frac{\mathrm{rot}(\Gamma_{\mathsf{f}})}{b}) + \mathrm{rot}(\Gamma_{\bar{\mathsf{f}}})(1-\frac{\mathrm{rot}(\Gamma_{\bar{\mathsf{f}}})}{b}) \\
& = & -2\cdot \mathrm{rot}(\Gamma_{\bar{\mathsf{f}}})(N-\mathrm{rot}(\Gamma_{\mathsf{f}})) + N\mathrm{rot}(\Gamma_{\bar{\mathsf{f}}}) - \mathrm{rot}(\Gamma_{\mathsf{f}}) \\
& \leq & N\mathrm{rot}(\Gamma_{\bar{\mathsf{f}}}) - \mathrm{rot}(\Gamma_{\mathsf{f}}).
\end{eqnarray*}
So 
\begin{equation}\label{eq-proof-bounds-MOY-tracks-2}
-\mathrm{rot}(\Gamma)(N+1-\frac{\mathrm{rot}(\Gamma)}{b}) \leq -\mathrm{rot}(\Gamma_{\mathsf{f}})(N-\frac{\mathrm{rot}(\Gamma_{\mathsf{f}})}{b}) - \mathrm{rot}(\Gamma_{\bar{\mathsf{f}}})(1-\frac{\mathrm{rot}(\Gamma_{\bar{\mathsf{f}}})}{b}) + N\mathrm{rot}(\Gamma_{\bar{\mathsf{f}}}) - \mathrm{rot}(\Gamma_{\mathsf{f}}).
\end{equation}
Similarly,
\begin{eqnarray*}
&& \mathrm{rot}(\Gamma)(N+1-\frac{\mathrm{rot}(\Gamma)}{b}) - \mathrm{rot}(\Gamma_{\mathsf{f}})(N-\frac{\mathrm{rot}(\Gamma_{\mathsf{f}})}{b}) - \mathrm{rot}(\Gamma_{\bar{\mathsf{f}}})(1-\frac{\mathrm{rot}(\Gamma_{\bar{\mathsf{f}}})}{b}) \\
& = & 2\cdot\mathrm{rot}(\Gamma_{\mathsf{f}})(1-\frac{\mathrm{rot}(\Gamma_{\bar{\mathsf{f}}})}{b}) + N\mathrm{rot}(\Gamma_{\bar{\mathsf{f}}}) - \mathrm{rot}(\Gamma_{\mathsf{f}}) \\
& \geq & N\mathrm{rot}(\Gamma_{\bar{\mathsf{f}}}) - \mathrm{rot}(\Gamma_{\mathsf{f}}).
\end{eqnarray*}
So
\begin{equation}\label{eq-proof-bounds-MOY-tracks-3}
\mathrm{rot}(\Gamma)(N+1-\frac{\mathrm{rot}(\Gamma)}{b}) \geq \mathrm{rot}(\Gamma_{\mathsf{f}})(N-\frac{\mathrm{rot}(\Gamma_{\mathsf{f}})}{b}) + \mathrm{rot}(\Gamma_{\bar{\mathsf{f}}})(1-\frac{\mathrm{rot}(\Gamma_{\bar{\mathsf{f}}})}{b}) + N\mathrm{rot}(\Gamma_{\bar{\mathsf{f}}}) - \mathrm{rot}(\Gamma_{\mathsf{f}}).
\end{equation}

For every $v \in V(\Gamma)$ as depicted in Figure \ref{fig-MOY-vertex}, we have 
\[
\rho_{\Gamma}(v) - \rho_{\Gamma_{\mathsf{f}}}(v) -\rho_{\Gamma_{\bar{\mathsf{f}}}}(v) = \frac{1}{2} (\mathsf{f}(e_1)\bar{\mathsf{f}}(e_2) + \bar{\mathsf{f}}(e_1)\mathsf{f}(e_2)) \geq 0.
\]
Recall that
\[
[v|\Gamma|\mathsf{f}] = \frac{1}{2} (\mathsf{f}(e_1)\bar{\mathsf{f}}(e_2) - \bar{\mathsf{f}}(e_1)\mathsf{f}(e_2)).
\]
This implies that
\begin{equation}\label{eq-proof-bounds-MOY-tracks-4}
-(\rho_{\Gamma}(v) - \rho_{\Gamma_{\mathsf{f}}}(v) -\rho_{\Gamma_{\bar{\mathsf{f}}}}(v)) \leq [v|\Gamma|\mathsf{f}] \leq \rho_{\Gamma}(v) - \rho_{\Gamma_{\mathsf{f}}}(v) -\rho_{\Gamma_{\bar{\mathsf{f}}}}(v).
\end{equation}

Putting \eqref{eq-proof-bounds-MOY-tracks-2}, \eqref{eq-proof-bounds-MOY-tracks-3} and \eqref{eq-proof-bounds-MOY-tracks-4} together, we get that
\begin{eqnarray*}
&& -\mathrm{rot}(\Gamma)(N+1-\frac{\mathrm{rot}(\Gamma)}{b}) - \rho(\Gamma) \\
& \leq & \sigma_{N,1}(\Gamma,\mathsf{f})  -\mathrm{rot}(\Gamma_{\mathsf{f}})(N-\frac{\mathrm{rot}(\Gamma_{\mathsf{f}})}{b})-\rho(\Gamma_{\mathsf{f}})  -\mathrm{rot}(\Gamma_{\bar{\mathsf{f}}})(1-\frac{\mathrm{rot}(\Gamma_{\bar{\mathsf{f}}})}{b})-\rho(\Gamma_{\bar{\mathsf{f}}})
\end{eqnarray*}
and 
\begin{eqnarray*}
&& \mathrm{rot}(\Gamma)(N+1-\frac{\mathrm{rot}(\Gamma)}{b}) + \rho(\Gamma) \\
& \geq & \sigma_{N,1}(\Gamma,\mathsf{f}) + \mathrm{rot}(\Gamma_{\mathsf{f}})(N-\frac{\mathrm{rot}(\Gamma_{\mathsf{f}})}{b}) + \rho(\Gamma_{\mathsf{f}}) + \mathrm{rot}(\Gamma_{\bar{\mathsf{f}}})(1-\frac{\mathrm{rot}(\Gamma_{\bar{\mathsf{f}}})}{b}) + \rho(\Gamma_{\bar{\mathsf{f}}}).
\end{eqnarray*}
But \eqref{eq-bounds-MOY-tracks} is true for $\left\langle \Gamma_{\mathsf{f}} \right\rangle_N$ and $\left\langle \Gamma_{\bar{\mathsf{f}}} \right\rangle_1$. That is,

\begin{eqnarray*}
 \min\deg_q \left\langle \Gamma_{\mathsf{f}} \right\rangle_N & \geq & -\mathrm{rot}(\Gamma_{\mathsf{f}})(N-\frac{\mathrm{rot}(\Gamma_{\mathsf{f}})}{b})-\rho(\Gamma_{\mathsf{f}}),\\
\max\deg_q \left\langle \Gamma_{\mathsf{f}} \right\rangle_N & \leq &  \mathrm{rot}(\Gamma_{\mathsf{f}})(N-\frac{\mathrm{rot}(\Gamma_{\mathsf{f}})}{b}) + \rho(\Gamma_{\mathsf{f}}), \\
\min\deg_q \left\langle \Gamma_{\bar{\mathsf{f}}} \right\rangle_1 & \geq &  -\mathrm{rot}(\Gamma_{\bar{\mathsf{f}}})(1-\frac{\mathrm{rot}(\Gamma_{\bar{\mathsf{f}}})}{b})-\rho(\Gamma_{\bar{\mathsf{f}}}),\\
\max\deg_q \left\langle \Gamma_{\bar{\mathsf{f}}} \right\rangle_1 & \leq &  \mathrm{rot}(\Gamma_{\bar{\mathsf{f}}})(1-\frac{\mathrm{rot}(\Gamma_{\bar{\mathsf{f}}})}{b}) + \rho(\Gamma_{\bar{\mathsf{f}}}).
\end{eqnarray*}

Thus, for any non-zero term $q^{\sigma_{N,1}(\Gamma,\mathsf{f})} \cdot \left\langle \Gamma_{\mathsf{f}} \right\rangle_N \cdot \left\langle \Gamma_{\bar{\mathsf{f}}} \right\rangle_1$ in \eqref{eq-proof-bounds-MOY-tracks-1}, we have

\begin{eqnarray*}
\min\deg_q (q^{\sigma_{N,1}(\Gamma,\mathsf{f})} \cdot \left\langle \Gamma_{\mathsf{f}} \right\rangle_N \cdot \left\langle \Gamma_{\bar{\mathsf{f}}} \right\rangle_1) & \geq & -\mathrm{rot}(\Gamma)(N+1-\frac{\mathrm{rot}(\Gamma)}{b}) - \rho(\Gamma), \\
\max\deg_q (q^{\sigma_{N,1}(\Gamma,\mathsf{f})} \cdot \left\langle \Gamma_{\mathsf{f}} \right\rangle_N \cdot \left\langle \Gamma_{\bar{\mathsf{f}}} \right\rangle_1) & \leq & \mathrm{rot}(\Gamma)(N+1-\frac{\mathrm{rot}(\Gamma)}{b}) + \rho(\Gamma).
\end{eqnarray*}
This shows that \eqref{eq-bounds-MOY-tracks} is true for $\left\langle \Gamma \right\rangle_{N+1}$ and completes the induction.
\end{proof}

\subsection{Proof of Proposition \ref{Prop-colored-MFW-sl-N}} Proposition \ref{Prop-colored-MFW-sl-N} is now a simple corollary of Proposition \ref{prop-bounds-MOY-tracks}. All we need to do is to keep track of the grading shifts used in the definition of the colored $\mathfrak{sl}(N)$ link homology.

\begin{proof}[Proof of Proposition \ref{Prop-colored-MFW-sl-N}]
Denote by $c_1,\dots,c_l$ the $l$ crossings of $B^{(m)}$ and by $\ve_i=\pm1$ the sign of the crossing $c_i$ for $i=1,\dots,l$. We say that replacing a crossing in Figure \ref{crossings-fig} by the shape in Figure \ref{MOY-res-fig} is a $k$-resolution of the crossing. Denote by $\Gamma_{k_1,\dots,k_l}$ the MOY resolution of $B^{(m)}$ obtained by applying the $k_i$ resolution on $c_i$ for $i=1,\dots,l$. Note that $\mathrm{rot}(\Gamma_{k_1,\dots,k_l}) = bm$. Moreover, each crossing $c_i$ in $B^{(m)}$ gives rise to four vertices in $\Gamma_{k_1,\dots,k_l}$. The $\rho_{\Gamma_{k_1,\dots,k_l}}$ value of these four vertices are $\frac{k_i m}{2}$, $\frac{k_i m}{2}$, $\frac{k_i (m-k_i)}{2}$ and $\frac{k_i (m-k_i)}{2}$. So 
\[
\rho(\Gamma_{k_1,\dots,k_l}) = \sum_{i=1}^l k_i(2m-k_i).
\]
Then, by Proposition \ref{prop-bounds-MOY-tracks}, we get that
\begin{eqnarray}
\label{eq-proof-prop11-1} \min\deg_q \left\langle \Gamma_{k_1,\dots,k_l} \right\rangle_N & \geq & -bm(N-m)-\sum_{i=1}^l k_i(2m-k_i),\\
\label{eq-proof-prop11-2} \max\deg_q \left\langle \Gamma_{k_1,\dots,k_l} \right\rangle_N & \leq & bm(N-m) + \sum_{i=1}^l k_i(2m-k_i).
\end{eqnarray}

By \eqref{eq-q-grading-def-+} and \eqref{eq-q-grading-def--}, we know that
\begin{equation}\label{eq-proof-prop11-3} 
\mathsf{s}_{q,N}(B^{(m)};\Gamma_{k_1,\dots,k_l}) = wm(N-m) + \sum_{i=1}^l \ve_ik_i.
\end{equation}
Since the graded dimension of $H_N(\Gamma_{k_1,\dots,k_l})$ is $\left\langle \Gamma_{k_1,\dots,k_l} \right\rangle_N$, using \eqref{eq-proof-prop11-1}, \eqref{eq-proof-prop11-2} and \eqref{eq-proof-prop11-3}, we get
\begin{eqnarray*}
&& \min\deg_q H_N(\Gamma_{k_1,\dots,k_l})\|\mathsf{s}_{h,N}(B^{(m)};\Gamma_{k_1,\dots,k_l})\|\{q^{\mathsf{s}_{q,N}(B^{(m)};\Gamma_{k_1,\dots,k_l})}\} \\
& \geq & (w-b)m(N-m) - \sum_{i=1}^l k_i(2m-k_i-\ve_i) \\
& \geq & (w-b)m(N-m) -lm^2+wm
\end{eqnarray*}
and 
\begin{eqnarray*}
&& \max\deg_q H_N(\Gamma_{k_1,\dots,k_l})\|\mathsf{s}_{h,N}(B^{(m)};\Gamma_{k_1,\dots,k_l})\|\{q^{\mathsf{s}_{q,N}(B^{(m)};\Gamma_{k_1,\dots,k_l})}\} \\
& \leq & (w+b)m(N-m) + \sum_{i=1}^l k_i(2m-k_i+\ve_i) \\
& \leq & (w+b)m(N-m) + lm^2+wm,
\end{eqnarray*}
where we also used the fact that, for any integer $k_i$,
\begin{eqnarray*}
-k_i(2m-k-\ve_i) & \geq & -m^2 +\ve_i m, \\
k_i(2m-k_i+\ve_i) & \leq & m^2+\ve_i m.
\end{eqnarray*}
By Theorem \ref{thm-sl-N-homology}, this implies Proposition \ref{Prop-colored-MFW-sl-N}.
\end{proof}

\end{document}